\documentclass[12pt]{article}
\usepackage{amsthm,amssymb,amsmath,graphicx,cite}
\usepackage[percent]{overpic}
\usepackage{algorithmic, algorithm}
\usepackage[margin = 1.2in]{geometry}

\newcommand{\D}{{\mathcal D}}
\newcommand{\ip}[2]{\left\langle #1 , #2 \right\rangle}    
\newcommand{\Oh}{{\mathcal O}}

\newcommand{\R}{{\mathbb R}}

\newcommand{\gcggap}{{\mathsf{CGgap}}}
\newcommand{\gmdgap}{{\mathsf{MDgap}}}
\newcommand{\Hgap}{{\mathsf{HYBgap}}}

\newcommand{\vertiii}[1]{{\vert\kern-0.25ex\vert\kern-0.25ex\vert #1 
    \vert\kern-0.25ex\vert\kern-0.25ex\vert}}

\DeclareMathOperator*{\argmin}{argmin}
\DeclareMathOperator*{\argmax}{argmax}

\DeclareMathOperator{\dom}{dom}

\newtheorem{theorem}{Theorem}
\newtheorem{proposition}{Proposition}

\newtheorem{corollary}{Corollary}
\newtheorem{definition}{Definition}
\newtheorem{assumption}{Assumption}

\title{Generalized conditional subgradient and generalized mirror descent: duality, convergence, and symmetry}
\author{Javier F. Pe\~na\thanks{Tepper School of Business,
Carnegie Mellon University, USA, {\tt jfp@andrew.cmu.edu}}}

\begin{document}

\maketitle

\begin{abstract} We provide new insight into a {\em generalized conditional subgradient} algorithm and a {\em generalized mirror descent} algorithm for the convex minimization problem
\[
\min_x \; \{f(Ax) + h(x)\}.\]
As Bach showed in [{\em SIAM J. Optim.}, 25 (2015), pp. 115--129], applying either of these two algorithms to this problem is equivalent to applying the other one to its Fenchel dual. We leverage this duality relationship to develop  new upper bounds and convergence results for the gap between the primal and dual iterates generated by these two algorithms.   We also propose a new {\em primal-dual hybrid} algorithm that combines features of the conditional subgradient and mirror descent algorithms to solve the primal and dual problems in a symmetric fashion.  Our algorithms and main results rely only on the availability of computable oracles for $\partial f$ and $\partial h^*$, and for $A$ and $A^*$.

\end{abstract}

\medskip


\medskip



\newpage

\section{Introduction}
Consider the convex minimization problem
\begin{equation}\label{eq.the.problem}
\min_{x\in X} \; \{f(Ax) + h(x)\} 
\end{equation}
where $A:X\rightarrow Y$ is a linear mapping between finite dimensional real vector spaces $X,Y$,  and $f:Y\rightarrow \R\cup\{\infty\}$ and 
$h:X\rightarrow \R\cup\{\infty\}$ are closed convex functions.  

We provide new insight into two natural algorithms for~\eqref{eq.the.problem}, namely a  {\em generalized conditional subgradient} algorithm and a {\em generalized mirror descent} algorithm.
These algorithms rely only on the availability of computable oracles for $\partial f$ and $\partial h^*$.  Our approach hinges on the interesting pairing established by Bach~\cite{Bach15} between these algorithms and the  problem~\eqref{eq.the.problem} and its Fenchel dual
\begin{equation}\label{eq.the.problem.dual}
\max_{u\in Y^*} \; \{-f^*(u) - h^*(-A^* u)\}. 
\end{equation}
Here $A^*:Y^*\rightarrow X^*$ is the adjoint of $A$, and $f^*:Y^* \rightarrow \R\cup\{\infty\}$ and $h^*:X^* \rightarrow \R\cup\{\infty\}$ are the Fenchel conjugates of $f$ and $h$ respectively~\cite{BorwL00,HiriL93,Rock70}.

As Bach~\cite{Bach15} showed, and as we detail in Section~\ref{sec.duality} below, applying the generalized conditional subgradient algorithm  to~\eqref{eq.the.problem} is equivalent to applying the generalized mirror descent algorithm to~\eqref{eq.the.problem.dual}.  Alternatively,
applying the generalized mirror descent algorithm to~\eqref{eq.the.problem} is equivalent to applying the generalized conditional subgradient algorithm to~\eqref{eq.the.problem.dual}.

 Our central results (Theorem~\ref{thm.main} and Theorem~\ref{thm.main.3}) take this duality relationship further.  These results give generic upper bounds on the gap between the primal and dual iterates generated by the generalized conditional subgradient, the generalized mirror descent, and a new {\em primal-dual hybrid} algorithm.  We subsequently leverage these  results to obtain an interesting generalization of the classic $\Oh(1/k)$ convergence rate of the conditional gradient algorithm~\cite{FreuG16,Jagg13}. More precisely, we show that when the step sizes are properly chosen, the duality gap between the primal and an average of dual iterates generated by the generalized conditional subgradient algorithm converges to zero at a rate $\Oh(1/k^{\gamma-1})$ provided $f$ satisfies a suitable $\gamma$-curvature condition relative to $h$ for some constant $\gamma > 1$.  We obtain analogous results for the generalized mirror descent and for the primal-dual hybrid algorithms. The classic $\Oh(1/k)$ rate corresponds to the special case $\gamma =2$.

\subsection{Positioning of the paper and related work}
This paper  sheds new light on the close duality connection between generalizations of two popular algorithmic schemes for problems of the form~\eqref{eq.the.problem}, namely the conditional gradient (also known as Frank-Wolfe) algorithm~\cite{FranW56,FreuG16,Jagg13} and the mirror descent algorithm~\cite{BeckT03,DuchSST10,NemiY83}.  
The conditional gradient and the mirror descent algorithms share the feature of not requiring any orthogonal projections.  This feature makes them attractive in a variety of applications where orthogonal projections are too costly or impractical but where subgradient oracles and Bregman projections are viable.  
Both the conditional gradient and mirror descent algorithms as well as numerous variants of them have been subjects of active research for several years.  Some of the many articles in this rapidly evolving literature include~\cite{BrauPTW18,BausBT16,Clar10,FreuG16,GidJL16,HarcJN15,Jagg13,LanZ16,LuFN18,Lu19,RaoSW15,Tebo18,YuZS17,YurtUTC17} as well as the many references therein.

Our work is inspired by and extends the ideas and results introduced by Bach~\cite{Bach15}, who showed the correspondence between a generalized conditional subgradient algorithm and a generalized mirror descent algorithm for~\eqref{eq.the.problem} and~\eqref{eq.the.problem.dual}.  Bach~\cite{Bach15} also showed convergence rates for both algorithms under certain strong convexity and Lipschitz assumptions.  In contrast to the approach followed by Bach~\cite{Bach15}, a main feature of our work is our focus on the gap between the primal and dual iterates.
Another main feature of our approach is the lack of reliance on any strong convexity or Lipschitz conditions.  Indeed, our algorithms and results make no references to any norms in $X$ or $Y$ at all.  Our approach enables us to give tighter and more general analyses of the generalized conditional subgradient, generalized mirror descent, and a new primal-dual hybrid algorithms under very general and mild assumptions.   Our duality gap approach also suggests a novel line-search strategy for selecting the step sizes in these three algorithms.  This strategy in turn gives an interesting generalization of the classical $\Oh(1/k)$ convergence rate of the conditional gradient algorithm~\cite{Jagg13,FreuG16}.  Our generalization relies on a new concept of {\em relative $\gamma$-curvature}.  This concept is a natural extension of the curvature constant introduced by Jaggi~\cite{Jagg13}.  It is also similar in spirit to the concepts of  {\em relative continuity} and {\em relative smoothness} as defined by  Lu~\cite{Lu19}, Lu et al.~\cite{LuFN18},  Bauschke et al.~\cite{BausBT16},  and Teboulle~\cite{Tebo18}.

Our algorithms and main results rely solely on the minimal conditions stated in Assumption~\ref{assum.blanket} in Section~\ref{sec.background}. 
This assumption concerns the availability of computable oracles for $\partial f$ and $\partial h^*$ and the compatibility of the ranges and domains of these oracles. Our convergence results rely on the additional mild Assumption~\ref{assum.bregman} concerning the computability of some generalized Bregman distances.  Our algorithms and results are readily invariant under invertible affine transformations of the spaces $X$ and $Y$.

\subsection{Organization of the paper}
The main sections of the paper are organized as follows.  Section~\ref{sec.background} describes some technical assumptions that we make throughout the paper.  This section also introduces the concept of {\em generalized Bregman distance} that plays a central role in this paper.
 Section~\ref{sec.algos} through Section~\ref{sec.convergence} present our main developments.  For exposition purposes, these sections consider the special case when $X=Y$ and $A:X\rightarrow X$ is the identity.  This simplification enables us to convey the gist of our developments more easily. Section~\ref{sec.extensions} describes how all of our developments extend to the more general problem~\eqref{eq.the.problem}.  
 
 In the special case when $X=Y$ and $A:X\rightarrow X$ is the identity, problem~\eqref{eq.the.problem} becomes
\begin{equation}\label{eq.primal}
\min_{x\in X} \; \{ f(x) + h(x) \}.
\end{equation}
Section~\ref{sec.algos} motivates and presents Algorithm~\ref{algo.gcg} and Algorithm~\ref{algo.gmd} which give descriptions of the generalized conditional subgradient and generalized mirror descent algorithms for problem~\eqref{eq.primal}.   Section~\ref{sec.duality} presents our core developments.  First, we detail the equivalence between applying Algorithm~\ref{algo.gcg} to~\eqref{eq.primal} and applying Algorithm~\ref{algo.gmd} to the Fenchel dual of~\eqref{eq.primal}.  We then introduce a new {\em primal-dual hybrid} algorithm, namely Algorithm~\ref{algo.hybrid}, that combines features of Algorithm~\ref{algo.gcg} and Algorithm~\ref{algo.gmd} in a perfectly symmetric fashion. Section~\ref{sec.duality} also presents our main results, namely Theorem~\ref{thm.main} and Theorem~\ref{thm.main.3}.  Theorem~\ref{thm.main} establishes a bound on the duality gap between the $(k+1)$-th primal iterate and a convex combination of the first $k$ dual iterates generated when Algorithm~\ref{algo.gcg} is applied to~\eqref{eq.primal}.  As Corollary~\ref{thm.main.2} states, an equivalent dual result readily follows when  Algorithm~\ref{algo.gmd} is applied to~\eqref{eq.primal}.  Theorem~\ref{thm.main.3} 
gives a similar bound on the gap between the primal and dual iterates generated by  Algorithm~\ref{algo.hybrid}.

Section~\ref{sec.convergence} leverages the results of Section~\ref{sec.duality} to bound the rate convergence to zero of the gap between primal and dual iterates generated when Algorithm~\ref{algo.gcg}, Algorithm~\ref{algo.gmd}, or Algorithm~\ref{algo.hybrid} is applied to~\eqref{eq.primal} and the step sizes are chosen judiciously.  Theorem~\ref{thm.gcg} shows that for Algorithm~\ref{algo.gcg} this gap converges to zero at a rate $\Oh(1/k^{\gamma -1})$ when $f$ satisfies a suitable $\gamma$-curvature condition relative to $h$ for $\gamma > 1$.  Corollary~\ref{thm.gmd} and Theorem~\ref{thm.hybrid} state analogous results for  Algorithm~\ref{algo.gmd} and Algorithm~\ref{algo.hybrid}.

Finally, Section~\ref{sec.extensions} shows how all of our developments extend to the more general problem~\eqref{eq.the.problem}.

\section{Technical background}
\label{sec.background}
We will rely on basic convex analysis machinery concerning convex functions, subgradients, Fenchel conjugate, and Fenchel duality as presented in the textbooks~\cite{BorwL00,HiriL93,Rock70}.

\subsection{Technical assumptions}
\label{sec.technical}
Throughout the paper we will make the following blanket assumption about the tuple $(X,Y,A,f,h)$.

\begin{assumption}\label{assum.blanket} {\em $X,Y$ are finite dimensional real vector spaces, $A:X\rightarrow Y$ is a linear mapping, and  $f:Y\rightarrow \R\cup\{\infty\}$ and $h:X\rightarrow \R\cup\{\infty\}$ are closed convex functions. 
There are available oracles that compute $x\mapsto Ax$ and $y\mapsto A^*u$ for all $x\in X, u\in Y^*$. Furthermore, for all $y\in \dom(\partial f)$ and $u \in \dom(\partial h^*)$ the following two conditions hold:
\begin{description}\item[(i)] there are available oracles that compute
\[
y\mapsto \partial f(y):=\argmax_{v \in Y^*} \{\ip{v}{y} - f^*(v)\}
\]
and
\[
u\mapsto \partial h^*(u):=\argmax_{x \in X} \{\ip{u}{x} - h(x)\},
\]
\item[(ii)] $-A^* \partial f(y) \in \dom(\partial h^*)$ and $A\partial h^*(u) \in \dom(\partial f)$. 
\end{description}
}
\end{assumption}

The line-search procedures in Section~\ref{sec.convergence} and Section~\ref{sec.extensions} will require the following additional mild assumption.

\begin{assumption}\label{assum.bregman} {\em The following {\em generalized Bregman distances}~\cite{Breg67,CensL81} are computable  for all $x,y\in\dom(\partial f)$ and $ u,v\in\dom(\partial h^*)$:
\[
D_f(y,x) = f(y) - f(x) - \ip{\partial f(x)}{y-x},
\]
and
\[
D_{h^*}(v,u) = h^*(v) - h^*(u) - \ip{v-u}{\partial h^*(u)}. 
\]}
\end{assumption}

Assumption~\ref{assum.blanket} ensures the well-posedness of problem~\eqref{eq.the.problem} and  also ensures that~\eqref{eq.the.problem} is amenable to the algorithms introduced in Section~\ref{sec.algos} and Section~\ref{sec.extensions} below.  In particular, Assumption~\ref{assum.blanket} implies that if $x\in \dom(\partial f \circ A)$ and $u\in -\dom(\partial h^*\circ A^*)$ then the points $x_{+} := \partial h^*(-A^*u)$ and $u_{+} := \partial f(Ax)$ are feasible for~\eqref{eq.the.problem} and~\eqref{eq.the.problem.dual} respectively, that is, $x_+\in  \dom(f\circ A) \cap \dom(h)$ and $u_+\in \dom(f^*)\cap (-\dom(h^*\circ A^*))$.   The mapping
\[
(x,u) \mapsto (\partial h^*(-A^*u), \partial f(Ax))
\]
is at the heart of the algorithms in Section~\ref{sec.algos} and Section~\ref{sec.extensions} below.

\subsection{Some notational convention}
We will rely on the following convenient notational convention.  When $x = \partial h^*(v)$ for some $v\in \dom(\partial h^*)$, we will write $\partial h(x)$ to denote $v$.  In this case, we will also write $D_h(y,x)$ to denote
\[
D_h(y,x):= h(y) - h(x) - \ip{v}{y-x} = h(y) - h(x) - \ip{\partial h(x)}{y-x}.
\]
In a symmetric fashion, when $u= \partial f(y)$ for some $y\in \dom(\partial f)$, we will  write $\partial f^*(u)$ to denote $y$ and $D_{f^*}(v,u)$ to denote
\[
D_{f^*}(v,u):= f^*(v) - f^*(u) - \ip{v-u}{y} = f^*(v) - f^*(u) - \ip{v-u}{\partial f^*(x)}.
\]
\subsection{Traditional conditional gradient context}
Assumption~\ref{assum.blanket} and Assumption~\ref{assum.bregman} readily hold in  the  usual set up of the conditional gradient algorithm~\cite{Jagg13,FreuG16}.  Consider the problem
\begin{equation}\label{eq.problem.simple}
\min_{x\in Q} \; f(x),
\end{equation}
where $f$ is differentiable on the compact convex set $Q \subseteq X$ and there is a linear   oracle that computes the support function
\[
u \mapsto \argmax_{x\in Q} \ip{u}{x}.
\]
Problem~\eqref{eq.problem.simple} can be written in the form~\eqref{eq.the.problem} by taking $h := \delta_{Q}$, the indicator function of the set $Q$, and $A:X\rightarrow X$ equal to the identity.  The  linear oracle for $Q$  corresponds to an oracle for $\partial h^*$ and the 
compactness of $Q$ implies that $\dom(\partial h^*) = X$.  Thus Assumption~\ref{assum.blanket} and Assumption~\ref{assum.bregman} readily hold in this case provided oracles for $f$ and $\nabla f$ are available.

\section{Generalized conditional subgradient and generalized mirror descent}
\label{sec.algos}
Suppose that $X=Y, \; A:X\rightarrow X$ is the identity mapping and suppose that the tuple $(X,Y,A,f,h)$ satisfies Assumption~\ref{assum.blanket}.

The thrust for this paper are two natural algorithmic schemes for solving~\eqref{eq.primal} via the available  oracles for $\partial f$ and $\partial h^*$.  The first scheme is a generalized version of the conditional gradient (also known as Frank-Wolfe) algorithm~\cite{FranW56,FreuG16,Jagg13}  based on the following update.  Given a current trial solution $x$ for~\eqref{eq.primal}, compute a new 
trial solution $x_+$ via 
\[
x_+ := (1-\alpha) x + \alpha s
\]
where $\alpha \in [0,1]$ and
\begin{align*}
s &:= \argmin_{y \in X} \{\ip{\partial f(x)}{y} + h(y)\} = \partial h^*(-\partial f(x)).
\end{align*}
The second scheme is a generalized mirror descent algorithm~\cite{BeckT03,DuchSST10,NemiY83} based on the following  update. Given a  current trial solution $y$ that satisfies $y = \partial h^*(v)$ for some $v\in \dom(h^*)$, compute a new trial solution $y_+$ via 
\begin{align*}
y_+ &= \argmin_{z\in X}\{\alpha\ip{\partial f(y) + v}{z} + D_h(z,y)\} \\
& = \argmin_{z\in X}\{\alpha\ip{\partial f(y) + \partial h(y)}{z} + D_h(z,y)\}.
\end{align*}
The generalized conditional subgradient update can be written as
\begin{equation}\label{eq.gcg}
x_+ = (1-\alpha) x + \alpha \partial h^*(-\partial f(x)).
\end{equation}
On the other hand, the generalized mirror descent update can be written as
\begin{align}\label{eq.gmd}
y_+ &= \argmin_{z\in X}\{\ip{\alpha\partial f(y) + (\alpha-1)\partial h(y)}{z} + h(z)\} \notag \\ &=\partial h^*((1-\alpha)\partial h(y) - \alpha \partial f(y)).
\end{align}
Observe the striking similarity between~\eqref{eq.gcg} and~\eqref{eq.gmd}.  This similarity is at the heart of our main developments.  As we detail below, the similarity between~\eqref{eq.gcg} and~\eqref{eq.gmd} underlies  the duality between the generalized conditional subgradient and the generalized mirror descent algorithms~\cite{Bach15}.

Algorithm~\ref{algo.gcg} and Algorithm~\ref{algo.gmd} give  descriptions of a generalized conditional subgradient and a generalized mirror descent algorithm for~\eqref{eq.primal} respectively.  It is easy to see that the iterates generated by Algorithm~\ref{algo.gcg} and by  Algorithm~\ref{algo.gmd} implement the update rules~\eqref{eq.gcg} and~\eqref{eq.gmd} respectively. 
We chose the descriptions in Algorithm~\ref{algo.gcg} and Algorithm~\ref{algo.gmd} to make the resemblance between the two algorithms more salient and to highlight that the algorithms only rely on the oracles for $\partial f$ and $\partial h^*$.  
\begin{algorithm}
\caption{Generalized conditional subgradient}\label{algo.gcg}
\begin{algorithmic}[1]
	\STATE {\bf input:}  $(f,h)$ and $x_0 \in \dom(\partial f)$
	\FOR{$k=0,1,2,\dots$}
	\STATE $u_k := \partial f(x_k)$
		\STATE $s_{k}:=\partial h^*(-u_k)$
		\STATE pick $\alpha_k \in [0,1]$ 
		\STATE $x_{k+1}:=(1-\alpha_k)x_k+\alpha_ks_k$
	\ENDFOR
\end{algorithmic}
\end{algorithm}

\begin{algorithm}
\caption{Generalized mirror descent}\label{algo.gmd}
\begin{algorithmic}[1]
	\STATE {\bf input:}  $(f, h)$ and  $v_0\in \dom(\partial h^*)$
	\FOR{$k=0,1,2,\dots$}
		\STATE $y_{k}:= \partial h^*(v_{k})$
		\STATE $z_k := \partial f(y_k)$ 
		\STATE pick $\alpha_k \in [0,1]$ 
		\STATE $v_{k+1}:= (1-\alpha_k)v_k-\alpha_kz_k$
	\ENDFOR
\end{algorithmic}
\end{algorithm}

\section{Duality and main results}
\label{sec.duality}
Again suppose that $X=Y, \; A:X\rightarrow X$ is the identity mapping and suppose that the tuple $(X,Y,A,f,h)$ satisfies Assumption~\ref{assum.blanket}.
Consider the Fenchel dual~\cite{BorwL00,HiriL93,Rock70} of~\eqref{eq.primal} 
\begin{equation}\label{eq.dual}
\max_{u\in X^*} \; \{ -f^*(u) - h^*(-u) \}, \end{equation}
which can be written as
\begin{equation}\label{eq.dual.again}
\min_{v\in X^*} \; \{h^*(v)+\tilde  f^*(v)\}
\end{equation}
for $\tilde f:X\rightarrow \R\cup\{\infty\}$ is defined via $\tilde f(y) := f(-y).$

Notice the nice symmetry between~\eqref{eq.primal} and~\eqref{eq.dual.again}.  The subgradient oracles for the pair $(f,h^*)$ are in one-to-one correspondence with subgradient oracles for the pair $(h^*,\tilde f^{**})$ respectively.  Thus, like problem~\eqref{eq.primal}, problem~\eqref{eq.dual.again} is amenable to both Algorithm~\ref{algo.gcg} and Algorithm~\ref{algo.gmd}.

Furthermore,  as shown by Bach~\cite{Bach15}, the following interesting duality between these two algorithms holds.  Running Algorithm~\ref{algo.gcg} on~\eqref{eq.primal} is identical to running Algorithm~\ref{algo.gmd} on~\eqref{eq.dual.again}.
More precisely, the iterates generated by
Algorithm~\ref{algo.gcg} applied to $(f,h)$ and started from $x_0\in\dom(\partial f)$ are the same, modulo some flipped signs, as those generated by Algorithm~\ref{algo.gmd} applied to $(h^*,\tilde f^*)$ 
and started from $v_0=-x_0\in\dom(\partial \tilde f^{**})$.  
Indeed, by letting $(v_k,y_k,z_k):=(-x_k,-u_k,s_k)$, the update at each iteration of Algorithm~\ref{algo.gcg}
\[
x_+ = (1-\alpha)x+\alpha s = (1-\alpha)x+\alpha \partial h^*(-u), \;\; u=\partial f(x)
\]
can be written as
\[
v_+ = (1-\alpha) v -\alpha z = (1-\alpha) v -\alpha \partial h^*(y), \;\; y = -\partial f(-v) = \partial \tilde f^{**}(v),
\]
which is exactly the update at each iteration of Algorithm~\ref{algo.gmd} applied to $(h^*, \tilde f^*)$.

Based on the above duality relationship, we propose a new {\em primal-dual hybrid} algorithm described in Algorithm~\ref{algo.hybrid} below.  Unlike Algorithm~\ref{algo.gcg} and Algorithm~\ref{algo.gmd}, Algorithm~\ref{algo.hybrid} is symmetric when applied to~\eqref{eq.primal} and~\eqref{eq.dual.again}.  That is, Algorithm~\ref{algo.hybrid} applied to $(f,h)$ and started from $x_0 \in \dom(\partial f), \; u_0\in -\dom(\partial h^*)$
generates the same iterates as it does when applied to $(h^*,\tilde f^*)$ and started from 
$-u_0\in \dom(\partial h^*), \, x_0 \in -\dom(\partial \tilde f^{**})$.

\begin{algorithm}
\caption{Primal-dual hybrid}\label{algo.hybrid}
\begin{algorithmic}[1]
	\STATE {\bf input:}  $(f, h)$ and  $x_0 \in \dom(\partial f), \; u_0\in -\dom(\partial h^*)$
	\FOR{$k=0,1,2,\dots$}
		\STATE $(s_{k},z_k):=(\partial h^*(-u_k),\partial f(x_k))$ 
		\STATE pick $\alpha_k \in [0,1]$ 		
		\STATE $(x_{k+1},u_{k+1}):=(1-\alpha_k)(x_k,u_k)+\alpha_k(s_k,z_k)$
	\ENDFOR
\end{algorithmic}
\end{algorithm}

\medskip

The weak duality relationship between~\eqref{eq.primal} and~\eqref{eq.dual}, which is well-known and easy  to show~\cite{BorwL00,HiriL93,Rock70}, is equivalent to the non-negativity of the  duality gap:
\[
f(x) + h(x) + f^*(u) + h^*(-u) \ge 0 \; 
\text{ for all } \; x\in X,\; u\in X^*.
\]
The next two theorems, which are the central results of this paper, provide {\em upper} bounds on the duality gap for the primal and dual iterates 
generated by Algorithm~\ref{algo.gcg}, Algorithm~\ref{algo.gmd}, and Algorithm~\ref{algo.hybrid}.  As we detail in Section~\ref{sec.convergence} below, Theorem~\ref{thm.main} and Theorem~\ref{thm.main.3} in turn imply that the duality gap converges to zero at a rate $\Oh(1/k^{\gamma-1})$ for $\gamma>1$ when the pair of functions $(f,h)$ satisfies a suitable relative $\gamma$-curvature conditions and the step sizes are judiciously chosen.  

The statements below will rely on the following double sequences $\lambda^k_i, \mu^k_i$ for $ k=1,2,\dots$ and $i=0,1,\dots,k-1$  determined by a sequence of step sizes $\alpha_k \in [0,1], \; k=0,1,\dots$.  For $k=0,1,2,\dots$ let
\begin{align}\label{eq.lm}
&\lambda^{k+1}_k = \alpha_k, \; \mu^{k+1}_k = 1\notag \\
&\lambda^{k+1}_i = (1-\alpha_k) \lambda^k_i, \; i=0,\dots,k-1 \\
&\mu^{k+1}_i = (1-\alpha_k) \mu^k_i, \; i=0,\dots,k-1. \notag
\end{align}
Observe that if $\alpha_0 = 1$ then for each $k=1,2,\dots$ we have $\lambda^k_i \ge 0,\; i=0,1,\dots,k-1$ and $\sum_{i=0}^{k-1} \lambda^k_i = 1$.

Our main statements will rely on the following notation.  For $x,s\in \dom(f)\cap \dom(h)$ and $\alpha \in [0,1]$ let
\[
\D_{f,h}(x,s,\alpha):=
D_f((1-\alpha)x+\alpha s,x)+ h((1-\alpha)x + \alpha s) - (1-\alpha)h(x) - \alpha h(s).
\]
The convexity of $h$ readily implies that
\[
\D_{f,h}(x,s,\alpha)\le
D_f((1-\alpha)x+\alpha s,x).
\]

\begin{theorem}\label{thm.main} Let $(x_k,u_k,s_k),\; k=0,1,2,\dots$ be the sequence of iterates generated by Algorithm~\ref{algo.gcg} applied to~\eqref{eq.primal}.  If $\alpha_0 = 1$ then for $k=1,2,\dots$ 
\begin{equation}\label{eq.main}
 \sum_{i=0}^{k-1} \lambda_i^k(f^*(u_i) + h^*(-u_i))
- \sum_{i=0}^{k-1} \mu_i^k \D_{f,h}(x_i,s_i,\alpha_i) =
-f(x_k) - h(x_k)
\end{equation}
where $\lambda_i^k, \mu_i^k$ are as in~\eqref{eq.lm}.
\end{theorem}
\begin{proof}  First, $u_k = \partial f(x_k)$ and $s_k = \partial h^*(-u_k)$ imply that
\begin{align}\label{ind.step}
f^*(u_k) + h^*(-u_k) 
&= \ip{u_k}{x_k-s_k} - f(x_k)-h(s_k).
\end{align}
We now prove~\eqref{eq.main} by induction.  
For $k=1$ we have  $x_1 = s_0= \partial h^*(-u_0)$ because $\alpha_0 = 1$.  Thus~\eqref{ind.step} implies
\begin{align*}
f^*(u_0) + h^*(-u_0) 
&= \ip{u_0}{x_0-x_1} -f(x_0) - h(x_1)\\
&=  - \ip{\partial f(x_0)}{x_1-x_0} - f(x_0) - h(x_1) \\
&= \D_{f,h}(x_0,s_0,1) - f(x_1) - h(x_1).
\end{align*} 
Hence~\eqref{eq.main} holds for $k=1$ since $\alpha_0=1$ and $\lambda^1_0=\mu^1_0=1$.

Suppose~\eqref{eq.main} holds for $k \ge 1$.  Adding up $(1-\alpha_k)$ times~\eqref{eq.main} plus $\alpha_k$ times~\eqref{ind.step}, and using~\eqref{eq.lm} and $x_{k+1} = (1-\alpha_k)x_k + \alpha_k s_k$ we obtain
\begin{align*}
&\sum_{i=0}^{k} \lambda_i^{k+1}(f^*(u_i) + h^*(-u_i))
-\sum_{i=0}^{k}\mu_i^{k+1} \D_{f,h}(x_i,s_i,\alpha_i) \\
&= -f(x_k) - (1-\alpha_k)h(x_k) - \alpha_k h(s_k) + \alpha_k\ip{u_k}{x_k-s_k} - \D_{f,h}(x_k,s_k,\alpha_k)\\
&= -f(x_k) - (1-\alpha_k)h(x_k) - \alpha_k h(s_k) + \ip{\partial f(x_k)}{x_k-s_k} - \D_{f,h}(x_k,s_k,\alpha_k)\\
&= -f(x_{k+1}) - h(x_{k+1}).
\end{align*}
Therefore~\eqref{eq.main} holds for $k+1$ as well. 
\end{proof}

\begin{corollary}\label{corol.thm.main}
 Let $(x_k,u_k,s_k),\; k=0,1,2,\dots$ be the sequence of iterates generated by Algorithm~\ref{algo.gcg} applied to~\eqref{eq.primal}.  If $\alpha_0 = 1$ then for $k=1,2,\dots$ 
\[
f(x_k) + h(x_k) + \sum_{i=0}^{k-1} \lambda_i^k(f^*(u_i) + h^*(-u_i)) 
\le \sum_{i=0}^{k-1} \mu_i^k D_{f}(x_{i+1},x_i).
\]
\end{corollary}
\begin{proof} This follows from Theorem~\ref{thm.main} and the fact that for $i=0,1,\dots,k-1$
\[
\D_{f,h}(x_i,s_i,\alpha_i) \le D_{f}((1-\alpha_i)x_{i} + \alpha_is_i,x_i) = D_{f}(x_{i+1},x_i).
\]
\end{proof}

The duality between~\eqref{eq.primal} and~\eqref{eq.dual} and between Algorithm~\ref{algo.gcg} and Algorithm~\ref{algo.gmd} automatically yields Corollary~\ref{thm.main.2}, which is a natural dual counterpart of Theorem~\ref{thm.main}.
Corollary~\ref{thm.main.2} relies on the following notation.  Observe that for $v,-z\in \dom(h^*)\cap \dom(\tilde f^*)$ and $\alpha \in [0,1]$ 
\begin{multline*}
\D_{h^*,\tilde f^*}(v,-z,\alpha)=
D_{h^*}((1-\alpha)v-\alpha z,v)+ \tilde f^*((1-\alpha)v - \alpha z) - (1-\alpha)\tilde f^*(v) - \alpha \tilde f^*(-z) \\
= D_{h^*}((1-\alpha)v-\alpha z,v)+  f^*(-(1-\alpha)v + \alpha z) - (1-\alpha) f^*(-v) - \alpha  f^*(z).
\end{multline*}
Once again, the convexity of $\tilde f^*$ implies that
\[
\D_{h^*,\tilde f^*}(v,-z,\alpha)\le
D_{h^*}((1-\alpha)v-\alpha z,v).
\]

\begin{corollary}\label{thm.main.2}
Let $(v_k,y_k,z_k),\; k=0,1,2,\dots$ be the sequence of iterates generated by Algorithm~\ref{algo.gmd} applied to~\eqref{eq.primal}.  If $\alpha_0 = 1$ then for $k=1,2,\dots$ 
\begin{equation}\label{eq.main.2}
 \sum_{i=0}^{k-1} \lambda_i^k(f(y_i) + h(y_i))
- \sum_{i=0}^{k-1} \mu_i^k \D_{h^*,\tilde f^*}(v_i,-z_i,\alpha_i) =
-f^*(-v_k) - h^*(v_k)
\end{equation}
where $\lambda_i^k, \mu_i^k$ are as in~\eqref{eq.lm}.  In particular, for $k=1,2,\dots$
\[
 \sum_{i=0}^{k-1} \lambda_i^k(f(y_i) + h(y_i)) + f^*(-v_k) + h^*(v_k)
\le \sum_{i=0}^{k-1} \mu_i^k D_{h^*}(v_{i+1},v_i).
\]
\end{corollary}

We also have the following symmetric analogue of Theorem~\ref{thm.main}.

\begin{theorem}\label{thm.main.3} 
 Let $(x_k, u_k, s_k, z_k),\; k=0,1,2,\dots$ be the sequence of iterates generated by
Algorithm~\ref{algo.hybrid}. If $\alpha_0=1$ then for $k=1,2,\dots$
\begin{multline}\label{eq.hybrid}
f(x_k) + h(x_k) + f^*(u_k) + h^*(-u_k) \\ = \sum_{i=0}^{k-1} 
\mu^k_i(\D_{f,h}(x_i,s_i,\alpha_i) + \D_{h^*,\tilde f^*}(-u_{i},-z_i,\alpha_i)),
\end{multline}
where $\mu_i^k$ are as in~\eqref{eq.lm}.
\end{theorem}
\begin{proof}
We proceed by induction.  For $k=1$ we have $x_1 = s_0 = \partial h^*(-u_0)$ and $u_1 = z_0 = \partial f(x_0)$ because $\alpha_0 = 1$.  Thus
\begin{align*}
f&(x_1) + h(x_1) + f^*(u_1) + h^*(-u_1) \\
&= f(x_1) - \ip{u_0}{x_1}- h^*(-u_0)  +\ip{u_1}{x_0} - f(x_0)+ h^*(-u_1)\\
&= f(x_1) - f(x_0)-\ip{u_1}{x_1-x_0} + h^*(-u_1) - h^*(-u_0)- \ip{u_0-u_1}{x_1}\\
&= f(x_1) - f(x_0)-\ip{\partial f(x_0)}{x_1-x_0} + h^*(-u_1) - h^*(-u_0)- \ip{u_0-u_1}{\partial h^*(-u_0)}\\
& = \D_{f,h}(x_0,s_0,1) + \D_{h^*,\tilde f^*}(-u_0,-z_0,1).
\end{align*}
Hence~\eqref{eq.hybrid} holds for $k=1$ since $\mu^1_0 = 1$ and $\alpha_0 = 1$.

\medskip

Suppose~\eqref{eq.hybrid} holds for $k\ge 1$.  Since $z_k = \partial f(x_k)$ and 
$s_k=\partial h^*(-u_k)$ and $(x_{k+1},u_{k+1}) = (1-\alpha_k)(x_k,u_k) + \alpha_k (s_k,z_k),$ it follows that
\begin{equation}\label{eq.pd.ind}
f(x_k) + h(s_k) + f^*(z_k) + h^*(-u_k) = \ip{z_k}{x_k} - \ip{u_k}{s_k},
\end{equation}
and 
\begin{multline}\label{eq.pd.ind.2}
f(x_{k+1}) - f(x_k) + h^*(-u_{k+1}) - h^*(-u_k) -
 D_f(x_{k+1},x_k) - D_{h^*}(-u_{k+1},-u_k) \\= - \alpha_k(\ip{z_k}{x_k} - \ip{u_k}{s_k}). 
\end{multline}
Next, adding up $(1-\alpha_k)$ times~\eqref{eq.hybrid}
plus $\alpha_k$ times~\eqref{eq.pd.ind} plus~\eqref{eq.pd.ind.2}, and using~\eqref{eq.lm} and $(x_{k+1},u_{k+1}) = (1-\alpha_k)(x_k,u_k) + \alpha_k (s_k,z_k)$
we get
\begin{multline*}
f(x_{k+1}) + h(x_{k+1}) + f^*(u_{k+1}) + h^*(-u_{k+1}) - \D_{f,h}(x_k,s_k,\alpha_k) - \D_{h^*,\tilde f^*}(-u_k,-z_k,\alpha_k)\\
=
\sum_{i=0}^{k-1} 
\mu^{k+1}_i(\D_{f,h}(x_{i},s_i,\alpha_i) + \D_{h^*,\tilde f^*}(-u_{i},-z_i,\alpha_i)).
\end{multline*}
Since $\mu^{k+1}_k = 1$, the previous equation can be rewritten as
\begin{multline*}
f(x_{k+1}) + h(x_{k+1}) + f^*(u_{k+1}) +
h^*(-u_{k+1}) \\
=
\sum_{i=0}^{k} 
\mu^{k+1}_i(\D_{f,h}(x_{i},s_i,\alpha_i) + \D_{h^*,\tilde f^*}(-u_{i},-z_i,\alpha_i)
).
\end{multline*}
Therefore~\eqref{eq.hybrid} holds for $k+1$ as well.  
\end{proof}

\begin{corollary}\label{corol.thm.main.3} 
 Let $(x_k, u_k, s_k, z_k),\; k=0,1,2,\dots$ be the sequence of iterates generated by
Algorithm~\ref{algo.hybrid}. If $\alpha_0=1$ then for $k=1,2,\dots$
\[
f(x_k) + h(x_k) + f^*(u_k) + h^*(-u_k) \le \sum_{i=0}^{k-1} 
\mu^k_i(D_f(x_{i+1},x_i) + D_{h^*}(-u_{i+1},-u_{i})).
\]
\end{corollary}
\section{Convergence results}
\label{sec.convergence}
Once again suppose that $X=Y, \; A:X\rightarrow X$ is the identity mapping and suppose that the tuple $(X,Y,A,f,h)$ satisfies Assumption~\ref{assum.blanket}.
We next leverage Corollary~\ref{corol.thm.main}, Corollary~\ref{thm.main.2}, and Corollary~\ref{corol.thm.main.3} to obtain some convergence results for Algorithm~\ref{algo.gcg}, Algorithm~\ref{algo.gmd}, and Algorithm~\ref{algo.hybrid} applied to~\eqref{eq.primal} and~\eqref{eq.dual}.

Corollary~\ref{corol.thm.main} readily implies that if $\alpha_0 = 1$ then the sequence of iterates $(x_k,u_k,s_k),\; k=0,1,2,\dots$ generated by Algorithm~\ref{algo.gcg} satisfies
\begin{equation}\label{eq.gap.gcg}
f(x_k) + g(x_k) + f^*(\hat u_k) + h^*(-\hat u_k) \le \gcggap_k,
\end{equation}
where
\begin{equation}\label{eq.hatu}
\hat u_k = \sum_{i=0}^{k-1} \lambda^k_{i} u_i \; \text{ or } \; \hat u_k = \argmin_{u_0,\dots,u_{k-1}} \{f^*(u_i)+h^*(-u_i)\},
\end{equation}
and $\gcggap_k, \; k=1,2,\dots$ is defined via $\gcggap_1 = D_f(s_0,x_0)$ and 
\begin{equation}\label{eq.gcggap}
\gcggap_{k+1} = (1-\alpha_k)\gcggap_k 
+ D_{f}((1-\alpha_k)x_{k}+\alpha_k s_k,x_k), 
\; k=1,2,\dots.
\end{equation}
Similarly, Corollary~\ref{thm.main.2} implies that if $\alpha_0 = 1$ then the sequence of iterates $(v_k,y_k,z_k),\; k=0,1,2,\dots$ generated by Algorithm~\ref{algo.gmd} satisfies
\begin{equation}\label{eq.gap.gmd}
f(\hat y_k) + g(\hat y_k) + f^*(-v_k) + h^*(v_k) \le \gmdgap_k,
\end{equation}
where
\begin{equation}\label{eq.haty}
\hat y_k = \sum_{i=0}^{k-1} \lambda^k_{i} y_i \; \text{ or } \; \hat y_k = \argmin_{y_0,\dots,y_{k-1}} \{f(y_i)+h(y_i)\},
\end{equation}
and $\gmdgap_k, \; k=1,2,\dots$ is defined via $\gmdgap_1 = D_{h^*}(-z_0,v_0)$ and 
\begin{equation}\label{eq.gmdgap}
\gmdgap_{k+1} = (1-\alpha_k)\gmdgap_k 
+ D_{h^*}((1-\alpha_k)v_{k}-\alpha_k z_k,v_k), 
\; k=1,2,\dots.
\end{equation}

On the other hand, Theorem~\ref{thm.main.3} implies that if $\alpha_0 = 1$ then the sequence of iterates $(x_k,u_k,s_k,z_k),\; k=0,1,2,\dots$ generated by Algorithm~\ref{algo.hybrid} satisfies
\begin{equation}\label{eq.gap.hybrid}
f(x_k) + h(x_k) + f^*(u_k) + h^*(-u_k) = \Hgap_k,
\end{equation}
where $\Hgap_1=D_f(s_0,x_0) + D_{h^*}(-z_0,-u_0)$ and
\begin{multline}\label{eq.hybridgap}
\Hgap_{k+1} = (1-\alpha_k)\Hgap_k 
+ D_{f}((1-\alpha_k)x_{k}+\alpha_ks_k,x_k)  \\ + D_{h^*}(-(1-\alpha_k)u_{k}-\alpha_kz_k,-u_k),
 \; \;  k=1,2,\dots.
\end{multline}

As the propositions below formally show, the above observations yield familiar $\Oh(1/k)$ convergence results for the popular  step size $\alpha_k = 2/(k+2), \; k=0,1,\dots$ provided a suitable {\em relative quadratic curvature} condition holds.

\begin{definition} {\em We say that $f$ has {\em quadratic curvature relative} to $h$ if there exists a finite constant $C$ such that for all $x\in\dom(\partial f)$  and $v\in \dom(\partial h^*)$  the following inequality holds for 
$s := \partial h^*(v)$ 
\begin{equation}\label{eq.quad.curv}
D_f(x+\alpha(s-x),x) \le \frac{C \alpha^2}{2} \text{ for all } \alpha \in [0,1].
\end{equation}
}
\end{definition}

This new concept of relative quadratic curvature condition is a generalization of the  curvature constant introduced by Jaggi~\cite{Jagg13}.  Indeed, consider a problem of the form~\eqref{eq.problem.simple} where $Q$ is compact and convex, $f$ is differentiable on $Q$, and a linear oracle for $Q$ is available.  In this context, Jaggi~\cite{Jagg13} defines the curvature constant of $f$ on $Q$ as follows
\[
C_{f,Q} = \sup_{x,s\in Q \atop \alpha\in (0,1]} \frac{D_f(x+\alpha(s-x),x)}{\alpha^2/2}.
\]
Observe that for $h = \delta_{Q}$, inequality~\eqref{eq.quad.curv} holds if $C \ge C_{f,Q}$.  We note that the smallest constant $C$ such that~\eqref{eq.quad.curv} holds for all $x\in Q$ and $s = \partial h^*(v), \; v\in \dom(\partial h^*)$ could be potentially smaller.  

The above concept of relative quadratic curvature is inspired by the concepts of relative smoothness and relative continuity introduced in~\cite{BausBT16,LuFN18,Lu19,Tebo18}.

\begin{proposition}\label{prop.gcg}
Suppose $f$ has quadratic curvature relative to $h$ with constant $C$.  If $\alpha_k = 2/(k+2), \; k=0,1,\dots$ then the sequence of iterates $(x_k,u_k,s_k),\; k=1,2,\dots$ generated by Algorithm~\ref{algo.gcg} satisfies
\begin{equation}\label{eq.conv}
f(x_k) + h(x_k) + f^*(\hat u_k) + h^*(-\hat u_k) \le \frac{2C}{k+2}
\end{equation}
where $\hat u_k$ is as in~\eqref{eq.hatu}.
\end{proposition}
\begin{proof} 
By~\eqref{eq.gap.gcg}, it suffices to show that for $k=1,2,\dots$
\begin{equation}\label{eq.gcggap.k}
\gcggap_k \le \frac{2C}{k+2}.
\end{equation}
We proceed by induction.  For $k=1$ inequality~\eqref{eq.quad.curv} and $\alpha_0=1$ imply that
\[
\gcggap_0 = D_f(x_1,x_0) = D_f(x_0+(s_0-x_0),x_0) \le \frac{C}{2} \le \frac{2C}{3}.
\]
Hence~\eqref{eq.gcggap.k} holds for $k=1$.
Suppose~\eqref{eq.gcggap.k} holds for $k\ge 1$.  Then~\eqref{eq.gcggap},~\eqref{eq.quad.curv}, and $\alpha_k = 2/(k+2)$ imply that 
\[
\gcggap_{k+1} \le \frac{k}{k+2}\cdot\frac{2C}{k+2} + \frac{2C}{(k+2)^2}
= \frac{2C(k+1)}{(k+2)^2}
\le \frac{2C}{k+3}.
\]
Therefore~\eqref{eq.gcggap.k} holds for $k+1$ as well.

\end{proof}


Again the duality between~\eqref{eq.primal} and~\eqref{eq.dual} and between Algorithm~\ref{algo.gcg} and Algorithm~\ref{algo.gmd} automatically yield the following corollary of Proposition~\ref{prop.gcg}.  Recall that $\tilde f$ is defined via $\tilde f(y)=f(-y)$.
\begin{corollary}\label{prop.gmd}
Suppose $ h^*$ has quadratic curvature relative to $\tilde f^*$ with constant $C^*$.  If $\alpha_k = 2/(k+2), \; k=0,1,\dots$ then the sequence of iterates $(v_k,y_k,z_k),\; k=0,1,2,\dots$ generated by Algorithm~\ref{algo.gmd} satisfies
\[
f(\hat y_k) + h(\hat y_k) + f^*(-v_k) + h^*(v_k) \le \frac{2C^*}{k+2}
\]
where $\hat y_k$ is as in~\eqref{eq.haty}.
\end{corollary}

The same inductive argument underlying the proof of Proposition~\ref{prop.gcg} together with~\eqref{eq.gap.hybrid} and~\eqref{eq.hybridgap} yields the following analogous result for Algorithm~\ref{algo.hybrid}.

\begin{proposition}\label{prop.hybrid}
Suppose $f$ has quadratic curvature relative to $h$  with constant $C$ and $ h^*$ has quadratic curvature relative to  $\tilde f^*$ with constant $C^*$.  If $\alpha_k =2/(k+2), \; k=0,1,\dots$ then the sequence of iterates $(x_k,u_k,s_k,z_k),\; k=0,1,2,\dots$ generated by Algorithm~\ref{algo.hybrid} satisfies
\[
f(x_k) + h(x_k) + f^*(u_k) + h^*(-u_k) \le \frac{2(C+C^*)}{k+2}.
\]
\end{proposition}

The identity~\eqref{eq.gcggap}  suggests the following {\em line-search} procedure to select the step size $\alpha_k$ in Algorithm~\ref{algo.gcg}:
\begin{equation}\label{eq.ls}
\alpha_k := \argmin_{\alpha\in[0,1]}  \left\{(1-\alpha) \gcggap_k + D_f(x_k+\alpha(s_k-x_k),x_k)\right\}.
\end{equation}
Likewise, the identity~\eqref{eq.gmdgap} suggests  the following {\em line-search} procedure to select the step size $\alpha_k$ in Algorithm~\ref{algo.gmd}:
\begin{align}\label{eq.ls.gmd}
\alpha_k 
&:= \argmin_{\alpha\in[0,1]}  \left\{(1-\alpha) \gmdgap_k + D_{h^*}(v_k-\alpha(z_k+v_k),v_k)\right\}.
\end{align}
Similarly, the identity~\eqref{eq.hybridgap} suggests  the following {\em line-search} procedure to select the step size $\alpha_k$ in Algorithm~\ref{algo.hybrid}:
\begin{align}\label{eq.ls.hybrid}
\alpha_k := \argmin_{\alpha\in[0,1]} \{(1-\alpha)\Hgap_k &+ D_f(x_k + \alpha(s_k-x_k),x_k) \notag \\ &+ D_{h^*}(-u_k - \alpha(z_k-u_k),-u_k) \}. 
\end{align}

The above line-search procedures are computable via  binary search provided Assumption~\ref{assum.bregman} holds.  These line-search procedures enable us to prove the convergence of Algorithm~\ref{algo.gcg}, Algorithm~\ref{algo.gmd}, and Algorithm~\ref{algo.hybrid} under the following more general $\gamma$-curvature condition.  We should note that under the stronger assumption

\begin{definition} {\em Let $\gamma > 1$. We  say that $f$ has {\em $\gamma$-curvature relative} to $h$ if there exists a finite constant $C$ such that for all 
$x\in\dom(f)$  and $v\in \dom(\partial h^*)$  the following inequality holds for $s = \partial h^*(v)$ 
\begin{equation}\label{eq.gamma.curv}
D_f(x+\alpha(s-x),x) \le \frac{C \alpha^\gamma}{\gamma} \text{ for all } \alpha \in [0,1].
\end{equation}
}
\end{definition}

We have the following interesting generalization of Proposition~\ref{prop.gcg}.

\begin{theorem}\label{thm.gcg}
Suppose $\gamma>1$ is such that $f$ has $\gamma$-curvature relative to $h$ with constant $C$.  If $\alpha_0 = 1$ and $\alpha_k \in [0,1], \; k=1,2,\dots$ is chosen via~\eqref{eq.ls} then the sequence of iterates $(x_k,u_k,s_k),\; k=0,1,2,\dots$ generated by Algorithm~\ref{algo.gcg} satisfies
\begin{equation}\label{eq.gcg.conv}
f(x_k) + g(x_k) + f^*(\hat u_k) + h^*(-\hat u_k) \le C\left(\frac{\gamma }{k+\gamma}\right)^{\gamma-1},
\end{equation}
where $\hat u_k$ is as in~\eqref{eq.hatu}.
\end{theorem}
\begin{proof} 
By~\eqref{eq.gap.gcg} it suffices to show that for $k=1,2,\dots$
\begin{equation}\label{eq.gamma.seq}
\gcggap_k \le C\left(\frac{\gamma}{k+\gamma}\right)^{\gamma -1}.
\end{equation}
We prove~\eqref{eq.gamma.seq} by induction on $k$.   For $k=1$ we have
\begin{equation}\label{eq.k.equal.1}
\gcggap_1 = D_f(x_1,x_0) \le \frac{C}{\gamma} \le C\left(\frac{\gamma}{1+\gamma}\right)^{\gamma -1}.
\end{equation}
where the last step follows from the weighted arithmetic mean geometric mean inequality. Hence~\eqref{eq.gamma.seq} holds for $k=1$.

\medskip

Suppose~\eqref{eq.gamma.seq} holds for $k\ge 1$.  Then~\eqref{eq.gamma.curv} and~\eqref{eq.ls} implies that for all $\alpha \in [0,1]$ 
\[
\gcggap_{k+1} \le (1-\alpha)\gcggap_{k} + \frac{C\alpha^\gamma}{\gamma}
\]
In particular, for $\alpha = \gamma/(k+\gamma)$ we have
\begin{align*}
\gcggap_{k+1} &\le C\frac{k}{k+\gamma}  \left(
\frac{\gamma}{k+\gamma}\right)^{\gamma -1}   
+ \frac{C \gamma^\gamma}{\gamma(k+\gamma)^\gamma} 
\\& 
= \frac{C(k+1)\gamma^{\gamma-1}}{(k+\gamma)^\gamma} \\&
 \le C\left(\frac{\gamma}{k+1+\gamma}\right)^{\gamma -1 },
\end{align*}
where the last step follows from the inequality
\[
(k+1)(k+1+\gamma)^{\gamma-1} \le (k+\gamma)^\gamma,
\]
which in turn follows from the weighted arithmetic geometric mean inequality.
Therefore~\eqref{eq.gamma.seq} holds for $k+1$ as well.
\end{proof}

Once again, the duality between~\eqref{eq.primal} and~\eqref{eq.dual} and between Algorithm~\ref{algo.gcg} and Algorithm~\ref{algo.gmd} automatically yields the following dual counterpart of Theorem~\ref{thm.gcg}.

\begin{corollary}\label{thm.gmd}
Suppose $\gamma > 1$ is such that $h^*$ has $\gamma$-curvature relative to $\tilde f^*$ with constant $C^*$.  If $\alpha_0 = 1$ and $\alpha_k \in [0,1], \; k=1,2,\dots$ is chosen via~\eqref{eq.ls.gmd} then the sequence of iterates $(v_k,y_k,z_k),\; k=0,1,2,\dots$ generated by Algorithm~\ref{algo.gmd} satisfies
\[
f(\hat y_k) + g(\hat y_k) + f^*(-v_k) + h^*(v_k) \le C^*\left(\frac{\gamma }{k+\gamma}\right)^{\gamma-1}
\]
where $\hat y_k$ is as in~\eqref{eq.haty}.
\end{corollary}

We also have the  following analogue of Theorem~\ref{thm.gcg} for Algorithm~\ref{algo.hybrid}.  We omit the proof of Theorem~\ref{thm.hybrid} since it is a straightforward extension of the proof of Theorem~\ref{thm.gcg}.

\begin{theorem}\label{thm.hybrid}
Suppose $\gamma>1$ is such that $f$ has $\gamma$-curvature relative to $h$  with constant $C$ and $h^*$ has $\gamma$-curvature relative to $\tilde f^*$ with constant $C^*$.  If $\alpha_k \in [0,1], \; k=0,1,\dots$ are chosen via~\eqref{eq.ls.hybrid}  then the sequence of iterates $x_k,u_k,\; k=0,1,2,\dots$ generated by Algorithm~\ref{algo.hybrid} satisfies
\[
f(x_k) + g(x_k) + f^*(u_k) + h^*(-u_k) \le (C+C^*)\left(\frac{\gamma }{k+\gamma}\right)^{\gamma-1}.
\]
\end{theorem}

The proof of Theorem~\ref{thm.gcg} readily shows that~\eqref{eq.gcg.conv} holds if $\alpha_k$ is chosen as $\alpha_k = \gamma/(k+\gamma)$.  However, this requires  knowledge of $\gamma>1$ which is unrealistic and could be too conservative.  A similar bound holds if instead $\alpha_k$ is chosen via the following approximate and more realistic line-search procedure.  Let $\delta\in (0,1)$ be a small fixed constant and choose $\alpha_k = \gamma_k/(k+\gamma_k)$ where $\gamma_k$ is such that $\gamma_k\ge \gamma-\delta$. This can be easily done via binary search as long as Assumption~\ref{assum.bregman} holds. The proof of Theorem~\ref{thm.gcg} shows that in this case the following modified version of~\eqref{eq.gcg.conv} holds
\[
f(x_k) + g(x_k) + f^*(\hat u_k) + h^*(-\hat u_k) \le C\left(\frac{\gamma-\delta }{k+\gamma-\delta}\right)^{\gamma-\delta-1}.
\]
The same considerations apply to the bounds in Corollary~\ref{thm.gmd} and Theorem~\ref{thm.hybrid}.

\medskip

We conclude this section by revisiting the role of Assumption~\ref{assum.bregman}, that is, the computability of $D_f(\cdot,\cdot)$ and $D_{h^*}(\cdot,\cdot)$. 
  As we already noted, this assumption is critical to ensure the viability of the line-search procedures~\eqref{eq.ls},~\eqref{eq.ls.gmd}, and~\eqref{eq.ls.hybrid}.  The results in this section can be sharpened under a stronger assumption as we next explain.
  If $D_{f^*}(\cdot,\cdot)$ and $D_{h}(\cdot,\cdot)$ are also computable then so are $\D_{f,h}(\cdot,\cdot,\cdot)$ and $\D_{h^*,f^*}(\cdot,\cdot,\cdot)$.  In that case the quantities $\gcggap_k, \gmdgap_k,\Hgap_k$ 
  can be sharpened by replacing~\eqref{eq.gcggap},~\eqref{eq.gmdgap}, and~\eqref{eq.hybridgap} with
  \[
 \gcggap_{k+1} = (1-\alpha_k)\gcggap_k + \D_{f,h}(x_{k},s_k,\alpha_k), \; k=1,2,\dots. 
  \]
  \[
  \gmdgap_{k+1} = (1-\alpha_k)\gmdgap_k + \D_{h^*,\tilde f^*}(v_{k},-z_k,\alpha_k), \; k=1,2,\dots
  \]
  and
  \[
  \Hgap_{k+1} = (1-\alpha_k)\Hgap_k + \D_{f,h}(x_{k},s_k,\alpha_k) + \D_{h^*,\tilde f^*}(-u_{k},-z_k,\alpha_k), \; k = 1,2,\dots
  \]
  respectively.  The line-search procedures ~\eqref{eq.ls},~\eqref{eq.ls.gmd}, and~\eqref{eq.ls.hybrid} can be  sharpened similarly.

\section{Extension to the general format}
\label{sec.extensions}
Suppose that the tuple $(X,Y,A,f,h)$  satisfies Assumption~\ref{assum.blanket}.
We next discuss how all of our previous developments extend to problems in the more general format
\begin{equation}\label{eq.primal.gral}
\min_{x\in X} \; \{ f(Ax) + h(x) \}
\end{equation}


Algorithm~\ref{algo.gcg}, Algorithm~\ref{algo.gmd}, and Algorithm~\ref{algo.hybrid} extend to~\eqref{eq.primal.gral} as detailed in Algorithm~\ref{algo.gcg.gral}, Algorithm~\ref{algo.gmd.gral}, and Algorithm~\ref{algo.hybrid.gral} respectively.   Furthermore,  
applying Algorithm~\ref{algo.gcg.gral} (Algorithm~\ref{algo.gmd.gral}) to~\eqref{eq.primal.gral} is equivalent to applying Algorithm~\ref{algo.gmd.gral} (Algorithm~\ref{algo.gcg.gral}) to its Fenchel dual
\[
\max_{u\in Y^*} \; \{-f^*(u) - h^*(-A^* u) \},
\]
which can be written as
\[
\min_{v\in Y^*} \; \{ h^*(A^* v) + \tilde f^*(v)\}
\]
for $\tilde f: Y\rightarrow \R\cup\{\infty\}$ defined via $\tilde f(y) := f(-y)$.

\begin{algorithm}
\caption{Generalized conditional subgradient, version 2}\label{algo.gcg.gral}
\begin{algorithmic}[1]
	\STATE {\bf input:}  $(f,h,A)$ and $x_0 \in \dom(\partial f\circ A)$
	\FOR{$k=0,1,2,\dots$}
	    \STATE $u_k:= \partial f(Ax_k)$
		\STATE $s_{k}:=\partial h^*(-A^* u_k)$
		\STATE pick $\alpha_k \in [0,1]$ 
		\STATE $x_{k+1}:=(1-\alpha_k)x_k+\alpha_ks_k$
	\ENDFOR
\end{algorithmic}
\end{algorithm}

\begin{algorithm}
\caption{Generalized mirror descent, version 2}\label{algo.gmd.gral}
\begin{algorithmic}[1]
	\STATE {\bf input:} $(f,h,A)$ and  $v_0\in \dom(\partial h^* \circ A^*)$
	\FOR{$k=0,1,2,\dots$}
		\STATE $y_{k}:= \partial h^*(A^* v_{k})$
		\STATE $z_k := \partial f(Ay_k)$ 
		\STATE pick $\alpha_k \in [0,1]$ 
		\STATE $v_{k+1}:= (1-\alpha_k)v_k-\alpha_kz_k$
	\ENDFOR
\end{algorithmic}
\end{algorithm}

\begin{algorithm}
\caption{Primal-dual hybrid, version 2}\label{algo.hybrid.gral}
\begin{algorithmic}[1]
	\STATE {\bf input:} $(f,h,A)$ and  $x_0 \in \dom(\partial f \circ A), \; u_0\in -\dom(\partial h^*\circ A^*)$
	\FOR{$k=0,1,2,\dots$}
		\STATE $(s_{k},z_k):=(\partial h^*(-A^* u_k), \partial f(Ax_k))$ 
		\STATE pick $\alpha_k \in [0,1]$ 		
		\STATE $(x_{k+1},u_{k+1}):=(1-\alpha_k)(x_k,u_k)+\alpha_k(s_k,z_k)$
	\ENDFOR
\end{algorithmic}
\end{algorithm}

\bigskip

Again by weak duality~\cite{BorwL00,HiriL93,Rock70} the duality gap is non-negative:
\[
f(Ax) + h(x) + f^*(u) + h^*(-A^*u) \ge 0 \; \text{ for all } \; x\in X, u\in Y^*.
\]
Proposition~\ref{thm.main.gral}, Proposition~\ref{thm.main.2.gral}, and Proposition~\ref{thm.main.3.gral} below give upper bounds on this duality gap for the primal-dual iterates generated by Algorithm~\ref{algo.gcg.gral}, Algorithm~\ref{algo.gmd.gral}, and Algorithm~\ref{algo.hybrid.gral}.  These propositions are extensions of Theorem~\ref{thm.main}, Corollary~\ref{thm.main.2},  and Theorem~\ref{thm.main.3} respectively.  We will rely on the following notation.  Observe that for $x,s\in \dom(f)$ and $\alpha \in [0,1]$ we have
\[
\D_{f\circ A,h}(x,s,\alpha):=
D_{f\circ A}((1-\alpha)x+\alpha s,x)+ h((1-\alpha)x + \alpha s) - (1-\alpha)h(x) - \alpha h(s).
\]
Similarly, for $v,z\in \dom(h)$ and $\alpha \in [0,1]$ we have
\[
\D_{h^*\circ A^*,\tilde f^*}(v,-z,\alpha)=
D_{h^*\circ A^*}((1-\alpha)v+\alpha z,v)+ f^*(-(1-\alpha)v + \alpha z) - (1-\alpha)f^*(-v) - \alpha f^*(z).
\]

\begin{proposition}\label{thm.main.gral}
Let $(x_k, u_k,s_k), \; k=0,1,2,\dots$ be the sequence of iterates generated by Algorithm~\ref{algo.gcg.gral} applied to~\eqref{eq.primal.gral}.  If $\alpha_0 = 1$ then for $k=1,2,\dots$ 
\[
\sum_{i=0}^{k-1} \lambda_i^k(f^*(u_i) + h^*(-A^* u_i))
- \sum_{i=0}^{k-1}\mu_i^k \D_{f\circ A,h}(x_{i},s_{i},\alpha_i) = -f(Ax_k) - h(x_k) 
\]
where $\lambda_i^k, \mu_i^k$ are as in~\eqref{eq.lm}.
\end{proposition}

\begin{proposition}\label{thm.main.2.gral} Let $(v_k,y_k,z_k),\; k=0,1,2,\dots$ be the sequence of iterates generated by Algorithm~\ref{algo.gmd.gral} applied to~\eqref{eq.primal.gral}.  If $\alpha_0 = 1$ then for $k=1,2,\dots$ 
\[
\sum_{i=0}^{k-1} \lambda_i^k(f(Ay_i) + h(y_i))
-\sum_{i=0}^{k-1}\mu_i^k \D_{h^*\circ A^*,\tilde f^*}(v_{i},-z_i,\alpha_i) 
=  -f^*(-v_k) - h^*(A^* v_k), 
\]
where $\lambda_i^k, \mu_i^k$ are as in~\eqref{eq.lm}.
\end{proposition}

\begin{proposition}\label{thm.main.3.gral} Let $(x_k, u_k, s_k, z_k),\; k=0,1,2,\dots$ be the sequence of iterates generated by Algorithm~\ref{algo.hybrid.gral} applied to~\eqref{eq.primal.gral}.  If $\alpha_0 = 1$ then for $k=1,2,\dots$ 
\begin{align*}
f(Ax_k) + h(x_k) &+f^*(u_k) + h^*(-A^* u_k) \\
&= \sum_{i=0}^{k-1}\mu_i^k (\D_{f\circ A,h}(x_{i},s_i,\alpha_i)+\D_{h^*\circ A^*,\tilde f^*}(-u_{i},-z_i,\alpha_i)),
\end{align*}
where $\mu_i^k$ are as in~\eqref{eq.lm}.
\end{proposition}

We omit the proofs of Proposition~\ref{thm.main.gral}, Proposition~\ref{thm.main.2.gral}, and Proposition~\ref{thm.main.3.gral} since they are straightforward modifications of the proofs of 
Theorem~\ref{thm.main}, Corollary~\ref{thm.main.2}, and Theorem~\ref{thm.main.3}.

\medskip

The developments in Section~\ref{sec.convergence} also extend in a similar fashion.  
Proposition~\ref{thm.main.gral} implies that if $\alpha_0 = 1$ then the sequence of iterates $(x_k,u_k,s_k),\; k=0,1,2,\dots$ generated by Algorithm~\ref{algo.gcg.gral} satisfies
\[
f(Ax_k) + g(x_k) + f^*(\hat u_k) + h^*(-A^*\hat u_k) \le \gcggap_k,
\]
where
\begin{equation}\label{eq.hatu.gral}
\hat u_k = \sum_{i=0}^{k-1} \lambda^k_{i} u_i \; \text{ or } \; \hat u_k = \argmin_{u_0,\dots,u_{k-1}} \{f^*(u_i)+h^*(-A^* u_i)\},
\end{equation}
$\gcggap_1 = D_f(As_0,Ax_0),$ and 
\[
\gcggap_{k+1} = (1-\alpha_k)\gcggap_k + D_f(A((1-\alpha_k)x_k+\alpha_ks_k),Ax_k), \; k=1,2,\dots.
\]
Similarly, Proposition~\ref{thm.main.2.gral} implies that if $\alpha_0 = 1$ then the sequence of iterates $(v_k,y_k,z_k),\; k=0,1,2,\dots$ generated by Algorithm~\ref{algo.gmd} satisfies
\[
f(A\hat y_k) + g(\hat y_k) + f^*(-v_k) + h^*(A^* v_k) \le \gmdgap_k,
\]
where
\begin{equation}\label{eq.haty.gral}
\hat y_k = \sum_{i=0}^{k-1} \lambda^k_{i} y_i \; \text{ or } \; \hat y_k = \argmin_{y_0,\dots,y_{k-1}} \{f(Ay_i)+h(y_i)\},
\end{equation}
$\gmdgap_1 = D_{h^*}(-A^* z_0,A^* v_0),$ and 
\[
\gmdgap_{k+1} = (1-\alpha_k)\gmdgap_k + D_{h^*}(A^*((1-\alpha_k)v_k-\alpha_kz_k),A^* v_k), \; k=1,2,\dots.
\]
On the other hand, Proposition~\ref{thm.main.3.gral} implies that if $\alpha_0 = 1$ then the sequence of iterates $(x_k,u_k,s_k,z_k),\; k=0,1,2,\dots$ generated by Algorithm~\ref{algo.hybrid.gral} satisfies
\begin{equation}\label{eq.gap.hybrid.gral}
f(Ax_k) + h(x_k) + f^*(u_k) + h^*(-A^* u_k) \le \Hgap_k,
\end{equation}
where $\Hgap_1:=D_f(As_0,Ax_0) + D_{h^*}(-A^* z_0,-A^* u_0)$ and
\begin{multline*}
\Hgap_{k+1} := (1-\alpha_k)\Hgap_k \\ + D_{f}(A((1-\alpha_k)x_k+\alpha_ks_k),Ax_k) + D_{h^*}(-A^*((1-\alpha_k)u_k + \alpha_kz_k),-A^* u_k).
\end{multline*}
Consider the following procedures for step size selection.  For Algorithm~\ref{algo.gcg.gral}:
\begin{equation}\label{eq.ls.gral}
\alpha_k := \argmin_{\alpha\in[0,1]}  \left\{(1-\alpha) \gcggap_k + D_f(A(x_k+\alpha(s_k-x_k)),Ax_k)\right\}.
\end{equation}
For Algorithm~\ref{algo.gmd.gral}:
\begin{equation}\label{eq.ls.gmd.gral}
\alpha_k := \argmin_{\alpha\in[0,1]}  \left\{(1-\alpha) \gmdgap_k + D_{h^*}(A^*(v_k-\alpha(z_k+v_k)),A^* v_k)\right\}.
\end{equation}
Finally, for Algorithm~\ref{algo.hybrid.gral}:
\begin{align}\label{eq.ls.hybrid.gral}
\alpha_k := \argmin_{\alpha\in[0,1]} \{(1-\alpha)\Hgap_k &+ D_f(Ax_k + \alpha(s_k-x_k),Ax_k) \notag \\ &+ D_{h^*}(-A^*(u_k + \alpha(z_k-u_k)),-A^* u_k) \}. 
\end{align}
Again the above line-search procedures are computable via  binary search provided Assumption~\ref{assum.bregman} holds.
Theorem~\ref{thm.gcg}, Corollary~\ref{thm.gmd}, and Theorem~\ref{thm.hybrid} extend as follows.
\begin{proposition}\label{thm.gcg.gral}
Suppose $\gamma > 1$ is such that $f\circ A$ has $\gamma$-curvature relative to $h$ with constant $C$.  If $\alpha_0 = 1$ and $\alpha_k \in [0,1], \; k=1,2,\dots$ is chosen via~\eqref{eq.ls.gral} then the sequence of iterates $(x_k,u_k,s_k),\; k=0,1,2,\dots$ generated by Algorithm~\ref{algo.gcg.gral} satisfies
\[
f(Ax_k) + g(x_k) + f^*(\hat u_k) + h^*(-A^*\hat u_k) \le C\left(\frac{\gamma }{k+\gamma}\right)^{\gamma-1}
\]
where $\hat u_k$ is as in~\eqref{eq.hatu.gral}.
\end{proposition}

\begin{proposition}\label{thm.gmd.gral}
Suppose $\gamma > 1$ is such that $ h^* \circ A^*$ has $\gamma$-curvature relative to $\tilde f^*$ with constant $C^*$.  If $\alpha_0 = 1$ and $\alpha_k \in [0,1], \; k=1,2,\dots$ is chosen via~\eqref{eq.ls.gmd.gral} then the sequence of iterates $(v_k,y_k,z_k),\; k=0,1,2,\dots$ generated by Algorithm~\ref{algo.gmd.gral} satisfies
\[
f(A\hat y_k) + h(\hat y_k) + f^*(-v_k) + h^*(A^* v_k) \le C^*\left(\frac{\gamma }{k+\gamma}\right)^{\gamma-1}
\]
where $\hat y_k$ is as in~\eqref{eq.haty.gral}.
\end{proposition}

\begin{proposition}\label{thm.hybrid.gral}
Suppose $\gamma > 1$ is such that $f\circ A$ has $\gamma$-curvature relative to $h$ with constant $C$ for some $\gamma > 1$ and 
$ h^* \circ  A^*$ has $\gamma$-curvature relative to $\tilde f^*$ with constant $C^*$.  If $\alpha_0 = 1$ and $\alpha_k \in [0,1], \; k=1,2,\dots$ is chosen via~\eqref{eq.ls.hybrid.gral} then the sequence of iterates $(x_k,u_k,s_k,z_k),\; k=0,1,2,\dots$ generated by Algorithm~\ref{algo.gcg.gral} satisfies
\[
f(Ax_k) + g(x_k) + f^*( u_k) + h^*(-A^* u_k) \le (C+C^*)\left(\frac{\gamma }{k+\gamma}\right)^{\gamma-1}.
\]
\end{proposition}

Again we omit the proofs of Proposition~\ref{thm.gcg.gral}, Proposition~\ref{thm.gmd.gral}, and Proposition~\ref{thm.hybrid.gral} since they are straightforward modifications of the proofs of 
Theorem~\ref{thm.gcg}, Corollary~\ref{thm.gmd}, and Theorem~\ref{thm.hybrid}.

\section{Conclusions}

We discussed three algorithms (Algorithm~\ref{algo.gcg.gral}, Algorithm~\ref{algo.gmd.gral}, and Algorithm~\ref{algo.hybrid.gral}) for problem~\eqref{eq.the.problem} and its Fenchel dual~\eqref{eq.the.problem.dual} that are based on the mapping
\[
(x,u) \mapsto (\partial h^*(-A^*u),\partial f(Ax)).
\]
Applying either Algorithm~\ref{algo.gcg.gral} or Algorithm~\ref{algo.gmd.gral} to~\eqref{eq.the.problem} is equivalent to applying the other one to~\eqref{eq.the.problem.dual}.
On the other hand, Algorithm~\ref{algo.hybrid.gral} treats~\eqref{eq.the.problem} and~\eqref{eq.the.problem.dual} in a completely symmetric fashion.

We established new upper bounds (Proposition~\ref{thm.main.gral}, Proposition~\ref{thm.main.2.gral}, and Proposition~\ref{thm.main.3.gral}) on the gap between
the primal and dual iterates generated by Algorithm~\ref{algo.gcg.gral}, Algorithm~\ref{algo.gmd.gral}, and Algorithm~\ref{algo.hybrid.gral}.  These bounds in turn imply that the duality gap converges to zero at a rate $\Oh(1/k^{\gamma-1})$ for $\gamma>1$  provided the functions $f,h$ satisfy some suitable $\gamma$-curvature conditions and the step sizes are judiciously chosen (Proposition~\ref{thm.gcg.gral}, Proposition~\ref{thm.gmd.gral}, and Proposition~\ref{thm.hybrid.gral}).

\bibliographystyle{plain}

\end{document}